\documentclass[11pt]{amsart}
\usepackage[dvips]{graphicx}
\usepackage{amssymb,amsmath,color, verbatim}
\usepackage{epic}
\usepackage{pstricks}
\usepackage{pst-node, mathrsfs}
\usepackage{url}

\usepackage[utf8]{inputenc}
\usepackage{pstricks-add}

\def\BBox{\kern  -0.2cm\hbox{\vrule width 0.2cm height 0.2cm}}

\newtheorem{lemma}{Lemma}[section]

\newtheorem{theorem}{Theorem}[section]

\newtheorem{remark}{Remark}[section]
\newtheorem{conjecture}[theorem]{Conjecture}

\title{Edge-girth-regular graphs arising from biaffine planes and Suzuki groups}

\author[Gabriela Araujo-Pardo]{Gabriela Araujo-Pardo}
\address{Gabriela Araujo-Pardo, Instituto de Matem\'aticas-Campus Juriquilla, Universidad Nacional Aut\'onoma de M\'exico, C.P. 076230, Boulevard Juriquilla \#3001, Juriquilla, Qro., México.}
\email{garaujo@matem.unam.mx}
\thanks{Research supported by PASPA-DGAPA-UNAM- M{\'exico} and CONACyT-{M\'exico} during a Sabbatical Stance in 2020, by CONACyT-M{\' e}xico under Project 282280 and PAPIIT-M{\' e}xico under Project IN101821.}

\author[Dimitri Leemans]{Dimitri Leemans}
\address{Dimitri Leemans, D\'epartement de Math\'ematique, Universit\'e libre de Bruxelles, C.P.216, Boulevard du Triomphe, 1050 Brussels, Belgium}
\email{Lemans.Dimitri@ulb.be}

\begin{document}

\date{}

\begin{abstract}
An edge-girth-regular graph $egr(v,k,g,\lambda)$, is a $k$-regular graph of order $v$, girth $g$ and with the property that each of its edges is contained in exactly $\lambda$ distinct $g$-cycles. An $egr(v,k,g,\lambda)$ is called extremal for the triple $(k,g,\lambda)$ if $v$ is the smallest order of any $egr(v,k,g,\lambda)$. 

In this paper, we introduce two families of edge-girth-regular graphs. The first one is a family of extremal $egr(2q^2,q,6,(q-1)^2(q-2))$ for any prime power $q\geq 3$ and, the second one is a family of $egr(q(q^2+1),q,5,\lambda)$ for $\lambda\geq q-1$ and $q\geq 8$ an odd power of $2$. In particular, if $q=8$ we have that $\lambda=q-1$.
Finally, we construct an $egr(32,5,5,12)$ and we prove that it is extremal.
\end{abstract}

\maketitle

\section{Introduction}

This paper deals with simple graphs, i.e. graphs with no loops and no multiple edges.

A graph $\mathcal G$ is {\em $k$-regular} if each of its vertices has exactly $k$ neighbours. It is of {\em girth} $g$ if its smallest cycles have $g$ vertices. The number of vertices of $\mathcal G$ is also called the {\em order} of $\mathcal G$.

In this paper we study a family of graphs called {\it{edge-girth-regular graphs}} that were first defined in \cite{JKM18}. These graphs are $k$-regular, of order $v$ and girth $g$. Moreover, they are such that each of their edges is contained in $\lambda$ distinct $g$-cycles (cycles of length $g$). They are denoted by $egr(v,k,g,\lambda)$ .

In the latter paper, the authors expose that their motivation is based on the fact that Moore cages (that are $k$-regular graphs of fixed girth $g$ having their number of vertices matching the Moore bound, for more information of this topic see \cite{EJ13}) are edge-girth-regular graphs. 

Throughout this paper, we use the natural lower bound for the order of a cage, called {\it {Moore's lower bound}} and denoted by $n_0(k,g)$. It is obtained by counting the vertices of a tree, ${\mathcal{T}}_{(g-1)/2}$, rooted on a vertex and with radius $(g-1)/2$, if $g$ is odd; or the vertices of a ``double-tree'' rooted at an edge (that is, two different ${\mathcal{T}}_{(g-3)/2}$ trees rooted each one at the vertices incident with an edge) if $g$ is even (see \cite{EJ13}):

  \begin{equation} n_0(k,g) = \left\{ \begin{array}{ll} 1 + \sum_{i=0}^{\frac{g-3}{2}} k(k-1)^{i} = \frac{k(k-1)^{\frac{g-1}{2}}-2}{k-2}, &\mbox{ if $g$ is odd};\\
2\sum_{i=0}^{\frac{g-2}{2}}(k-1)^{i}= \frac{2(k-1)^{\frac{g}{2}}-2}{k-2}, &\mbox{ if $g$ is
even}.\end{array}\right.\end{equation}

Moreover, the well-known {\emph{Cage Problem}} consists in finding $k$-regular graphs of fixed girth $g$ and minimum order, denoted by $n(k,g)$, that we call $(k;g)$-cages.

If $\mathcal G$ is a $(k;g)$-cage of order $n_0(k;g)$, it is a Moore Cage. If $\mathcal G$ is a $(k;g)$-graph of order $|V(\mathcal G)|$ then we call the {\emph{excess}} of $\mathcal G$ the difference $|V(\mathcal G)|-n_0(k;g)$. 
It is well known (see \cite{EJ13}) that there are very few graphs that attain the Moore Bound and the existence of such graphs is related with the existence of some specific finite geometries or generalized quadrangles (see \cite{BBR18,BI80,EJ13}). Consequently, the main challenge of this problem is not only to find small graphs, but also to prove that they are the smallest that exist and have order $n(k;g)$, or in other words that they are cages. 

In \cite{JKM18} the authors introduce the concept of edge-girth-regular graphs and give results about their structure. One of them implies the relationship between these graphs and the Moore cages. They analyze the $egr(v,3,g,\lambda)$-graphs for $3\leq g\leq 6$ and prove that $K_4$ is the unique $egr(4,3,3,2)$-graph, $K_{3,3}$ the unique $egr(6,3,4,4)$-graph and the cube $Q_3$ the unique $egr(4,3,4,2)$-graph. For girth $g=5$ they also prove that the Petersen graph\footnote{The Petersen graph is the cubic Moore cage of girth $5$.} and the skeleton graph of the dodecahedron are the unique $egr(10,3,5,4)$-graph and $egr(20,3,5,2)$-graph respectively; and, for girth six they give four $egr(v,3,6,\lambda)$-graphs, namely the Heawood graph which is the unique $egr(14,3,6,8)$-graph that is the cubic Moore Cage of girth six, the M\"{o}ebius graph that is an $egr(16,3,6,6)$-graph, and two $egr(v,3,6,4)$-graphs, the Pappus graph (which is the incidence graph of Pappus configuration) and Desargues graph (which is the incidence graph of Desargues configuration). The Pappus graph has order $18$ and the Desargues graph has order $20$. 
Also in \cite{JKM18} the authors give several constructions that produce infinite families of $egr$-graphs, as well as a general result of edge-girth-regular graphs for $\lambda=2$ and any regularity and girth, related with topological graph theory. 

Recently, Drglin, Filipovski, Jajcay and Raiman \cite{DFJR21} introduced a special class of these graphs called {\it{extremal edge-girth-regular graphs}}. These are $egr(v,k,g,\lambda)$ with $v=n(k,g,\lambda)$ where $n(k,g,\lambda)$ is the smallest order of a $(v,k,g,\lambda)$-graph fixing the triplet $(k,g,\lambda)$ (see~\cite{DFJR21}). As usual if we have an $egr(v,k,g,\lambda)$-graph $\mathcal G$, the {\emph{excess}} of $\mathcal G$ is $|V(\mathcal G)|$-$n(k,g,\lambda)$.

In the same paper, the authors provide lower bounds for $n(k,g,\lambda)$ depending of the parity of $g$ and related with the Moore bound for $k$-regular cages of girth $g$, denoted by $n_0(k,g)$. Specifically, they improve the lower bound of $n(k,g,\lambda)$ for bipartite graphs. Moreover, they give results for some special classes of extremal edge-girth-regular graphs with $\lambda=1$, called \emph{girth-tight} in \cite{PW07}, and some constructions though the concept of \emph{canonical double cover} used on topological graph theory, especially in order to obtain extremal edge-girth-regular graphs of girth $4$.

The purpose of our paper is to introduce new techniques to construct infinite families of $(v,k,g,\lambda)$-graphs for some specific values of $k\geq 4$, $g=\{5,6\}$ and $\lambda \geq 3$ (which do not appear in~\cite{JKM18,DFJR21}) and prove that some of these families are extremal. 

The paper is organized as follows.
In Section~\ref{egrgfg} we introduce a family of graphs given in \cite{B67} and we prove that, for $q$ a prime power they are $egr(2q^2,q,6,(q-1)^2(q-2))$.
Moreover, using the algebraic properties of the incidence graph of a projective plane and the graph described above for $q=4$, we construct an specific $(32,5,5,12)$-graph. 
In Section~\ref{egrgsg} we introduce a family of  $egr(q(q^2+1),q,5,\lambda)$-graphs with $\lambda\geq q-1$ and $q=2^{2\alpha+1}$ and $\alpha$ a strictly positive integer. In particular, if $q=8$ we have that $\lambda=q-1$; these graphs arise from the Suzuki simple groups.
Finally, in Section~\ref{extremalegrg} we introduce the bounds given in~\cite{DFJR21} and we prove that the graphs given in Section~\ref{sub1} and the graph constructed in Section~\ref{sub2} are extremal, in particular that $n(q,6,(q-1)^2(q-2))=2q^2$ and that $n(5,5,12)=32$. In particular, if we let $q=3$ in the family of $egr(2q^2,q,6,(q-1)^2(q-2))$-graphs, we obtain the Pappus graph that was mentioned in~\cite{DFJR21} to be extremal.

\section{Edge-girth-regular graphs from Moore cages of girth 6}\label{egrgfg}

Let $q$ be a prime power and consider the Levi graph or incidence graph $\mathcal B_q$ of elliptic semiplanes of type C that is obtained from the projective plane of order $q$ by choosing an incident point-line pair $(p,l)$ and deleting all the lines incident with $p$ and all the points belonging to $l$. Thus, the Levi (or incidence) graph
$\mathcal B_q$ is bipartite, $q$-regular and has $2q^2$ vertices, which correspond in the elliptic semiplane to $q^2$ points and $q^2$ lines 
both partitioned into $q$ parallel classes or blocks of $q$ elements each. It is 
 also known as the biaffine plane (see \cite{AABB17,B67}). The following description of $\mathcal B_q$ appears in \cite{ABMM16} and it was used to construct families of graphs and mixed graphs of girth $5$ or diameter $2$ (see \cite{AABB17,AABL12,ABMM16,ADG20}).
\\
The following Remark \ref{remark1} describes the biaffine plane and its incidence graph. 

\begin{remark}\label{remark1}\cite{ABMM16}.
Let $\mathbb{F}_q$ be the finite field of order $q$ with $q$ a power of a prime.
\begin{enumerate}
\item[(i)] Let $\mathscr{L} =\mathbb{F}_q \times \mathbb{F}_q$ and $\mathscr{P} =\mathbb{F}_q \times \mathbb{F}_q$. Denote the elements of $\mathscr{L}$ and $\mathscr{P}$ using ``brackets'' and ``parenthesis'', respectively.
The following set of $q^2$ lines define a biaffine plane:

\begin{equation}
[m, b] = \{(x, mx + b) : x \in \mathbb{F}_q\} \  for \ all\ m, b \in \mathbb{F}_q.
\end{equation}

\item[(ii)] The incidence graph of the biaffine plane is a bipartite graph $\mathcal B_q = (\mathscr{P},\mathscr{L})$ which is $q$-regular, has order $2q^2$, diameter 4 and girth 6, if $q \geq 3$; and girth 8, if $q=2$.

\item[(iii)] The vertices mutually at distance 4 are the elements of the sets $L_m = \{[m, b] : b \in \mathbb{F}_q\}$, and $P_x = \{(x, y) : y \in \mathbb{F}_q\}$ for all $x,m \in \mathbb{F}_q$.
\end{enumerate}
\end{remark}

\subsection{A family of $egr(2q^2,q,6,(q-1)^2(q-2))$-graphs for $q$ a prime power}\label{sub1}
In the following Theorem we prove that $\mathcal B_q$ is an edge-girth-regular graph. 
\begin{theorem}\label{B_q} For any $q\geq 3$, $q$ a prime power, $\mathcal B_q$ is an $egr(2q^2, q,6,(q-1)^2(q-2))$.
\end{theorem}

\begin{proof}
The order, degree and girth are given in Remark \ref{remark1}, hence we only have to prove that any edge is in exactly $(q-1)^2(q-2)$ $6$-cycles. 
Let $u_1u_2$ be any edge in $\mathcal B_q$, then we can suppose, without loss of generality that $u_1\in P_x$ and $u_2\in L_m$, for some $x,m\in \mathbb{F}_q$. We then have $(q-1)$ possible choices for a neighbour $u_3\neq u_1$ of $u_2$, one in each $P_x'$ for $x'\in \mathbb{F}_q$ and $x'\not= x$. Again, to select a neighbour $u_4$ of $u_3$ that is not $u_2$, we have $(q-1)$ options, one in each $L_m'$ for $m'\in \mathbb{F}_q$ and $m'\not= m$.

Until here, we have a path $(u_1,u_2,u_3,u_4)$. Finally if we select a new neighbor $u_5\neq u_3$ of $u_4$ in $P_x$, as $u_1\in P_x$, by item (iii) of Remark \ref{remark1} we have that then $d(u_1,u_5)=4$ and it is impossible than we close this path as a $6$-cycle. Hence we have only $(q-2)$ options to select $u_5$ in $P_{x''}$, with $x''\in \mathbb{F}_q$ and $x''\not= x,x'$. Then we have a path $(u_1,u_2,u_3,u_4,u_5)$ of length four and there is a unique way to close this path in a 6-cycle, namely by taking
$u_6\in L_t$ such that, if $u_1=(x,y)$ and $u_5=(x'',y'')$,  $t=\frac{y''-y}{x''-x}$.
Finally, $C=(u_1,u_2,u_3,u_4,u_5,u_6)$ is a $6$-cycle of $\mathcal B_q$. 

In total, we have $(q-1)^2(q-2)$ cycles of length $6$ in $\mathcal B_q$ that contain the edge $u_1u_2$.
\end{proof}

\subsection{The special graph $egr(32,5,5,12)$.}\label{sub2}

In the following theorem, we construct a special edge-girth-regular graph using $\mathcal B_4$.  

\begin{theorem}\label{thm:egrcage}
There exists an $egr(32,5,5,12)$.
\end{theorem}
\begin{proof}
Let $\mathcal B_4$ the incidence graph of the biaffine plane of order $4$ given in Remark \ref{remark1}. 

Let $\mathbb{F}_4$ be the Galois Field of order $4$ constructed as an extension field of $\mathbb{Z}_2$, that is, the elements of $\mathbb{F}_4$ are $\{0,1,\alpha,\alpha^2\}$, where $\alpha$ is a primitive root of the irreducible polynomial $f(x)=x^2+x+1$ over $\mathbb{Z}_2$. Recall that, in this case $\alpha^2=\alpha+1$.

Note that $\mathcal B_4$ has $32$ vertices distributed in $8$ sets of size $4$. The vertex-set $V(\mathcal B_4) = P_0 \cup P_1 \cup P_\alpha \cup P_{\alpha^2} \cup L_0 \cup L_1 \cup L_\alpha\cup L_{\alpha^2}$. The vertices in the sets $P_i$ for $i\in \mathbb{F}_4$ correspond to points of the biaffine plane and the vertices in the sets $L_j$ for $j\in \mathbb{F}_4$ correspond to lines. We call the vertices that correspond to points $P$-vertices and the vertices that correspond to lines $L$-vertices. 
Starting with $\mathcal B_4$ we construct a graph $\mathcal H$ with same vertex-set as $\mathcal B_4$ and the following set of edges.
\[ E(\mathcal H)=E(\mathcal B_4)\cup \mathcal{E}\]
where \[\mathcal{E}=\{(i,0)(i,1),(i,\alpha)(i,\alpha^2) | i\in \mathbb{F}_4\} \cup \{[j,0][j,\alpha],[j,1][j,\alpha^2] | j \in \mathbb{F}_4\}.\]

Note that, as $\mathcal B_4$ is $4$-regular and $\mathcal{E}$ is a matching, the graph $\mathcal H$ is a $5$-regular graph of order 32. It remains to show that $\mathcal H$ has girth $5$ and that $\lambda=12$. 

First we prove that $\mathcal H$ is a vertex-transitive graph (this proof is analogous to the one given in \cite{ANS06}). Note that the translations of the biaffine plane still act as automorphisms of $\mathcal H$, that is $\tau_{(a,b)}(i,j)=(i+a,j+b)$ and the same holds for the lines of $\mathcal B_4$. Moreover, let $\Phi_\alpha$ be an automorphism that exchanges points and lines and preserves incidences defined as $\Phi_\alpha(i,j)=[i,\alpha j]$ and $\Phi_\alpha[i,j]=(\alpha i,\alpha j).$ The $\tau(a,b)$'s show that all elements of $\mathcal P$ are in a same orbit and the elements of $\mathcal L$ are also in a same orbit. An element $\Phi_\alpha$ swaps elements of $\mathcal P$ with elements of $\mathcal L$. Hence there is a unique orbit on the vertices of $\mathcal B_4$ under the action of the automorphism group of $\mathcal B_4$ and the graph is vertex-transitive.

To prove that $\mathcal H$ has girth $5$, notice that as $\mathcal B_4$ has girth equal to $6$, we have to prove that by adding the edges of $\mathcal E$ to construct $\mathcal H$ we do not produce $3$-cycles or $4$-cycles but we produce 5-cycles. Clearly, $\mathcal H$ does not have $3$-cycles, because as $\mathcal B_4$ is bipartite, any $3$-cycle would have at least one edge in $\mathcal E$. But as $\mathcal E$ is a matching it should have at most one edge in $\mathcal E$. By construction if $e=uv\in \mathcal E$ we have that $N(u)\cap N(v)=\emptyset$, and we do not have  a $3$-cycle. As $\mathcal B_4$ has girth $6$, if $\mathcal H$ has a $4$-cycle, this cycle has edges in $\mathcal E$. Moreover as $\mathcal E$ is a matching we have to prove that if we have an edge $e=uv \in \mathcal E$ the neighbors of $u$ and $v$ on $\mathcal L$, called $u'$ and $v'$ respectively, do not induce an edge in $\mathcal E$. Let $uv$ be an edge between vertices of $\mathcal P$. Then $u=(x,y)$ and $v=(x,y+1)$ for some $x$ and $y\in \mathbb F_q$ and let $u'=[m,b]$ adjacent to $u$. By construction $y=mx+b$ and $y+1=mx+b+1$. Hence $v'=[m,b+1]$. But, by construction of $\mathcal H$ we have that $u'v'\notin \mathcal E$. Hence we do not have $4$-cycles. 
    
    Now, to prove that the girth is equal to $5$, we only have to give one $5$-cycle.  $\mathcal{C}=((0,0),(0,1),[1,1],(1,0),[0,0],(0,0)).$
    A part of the graph $\mathcal H$ is depicted in Figure \ref{special}. For clarity in the picture we do not draw all edges. The cycle $\mathcal{C}$ is depicted in red.
    
    \begin{figure}
\centering
\includegraphics[width=0.6\linewidth]{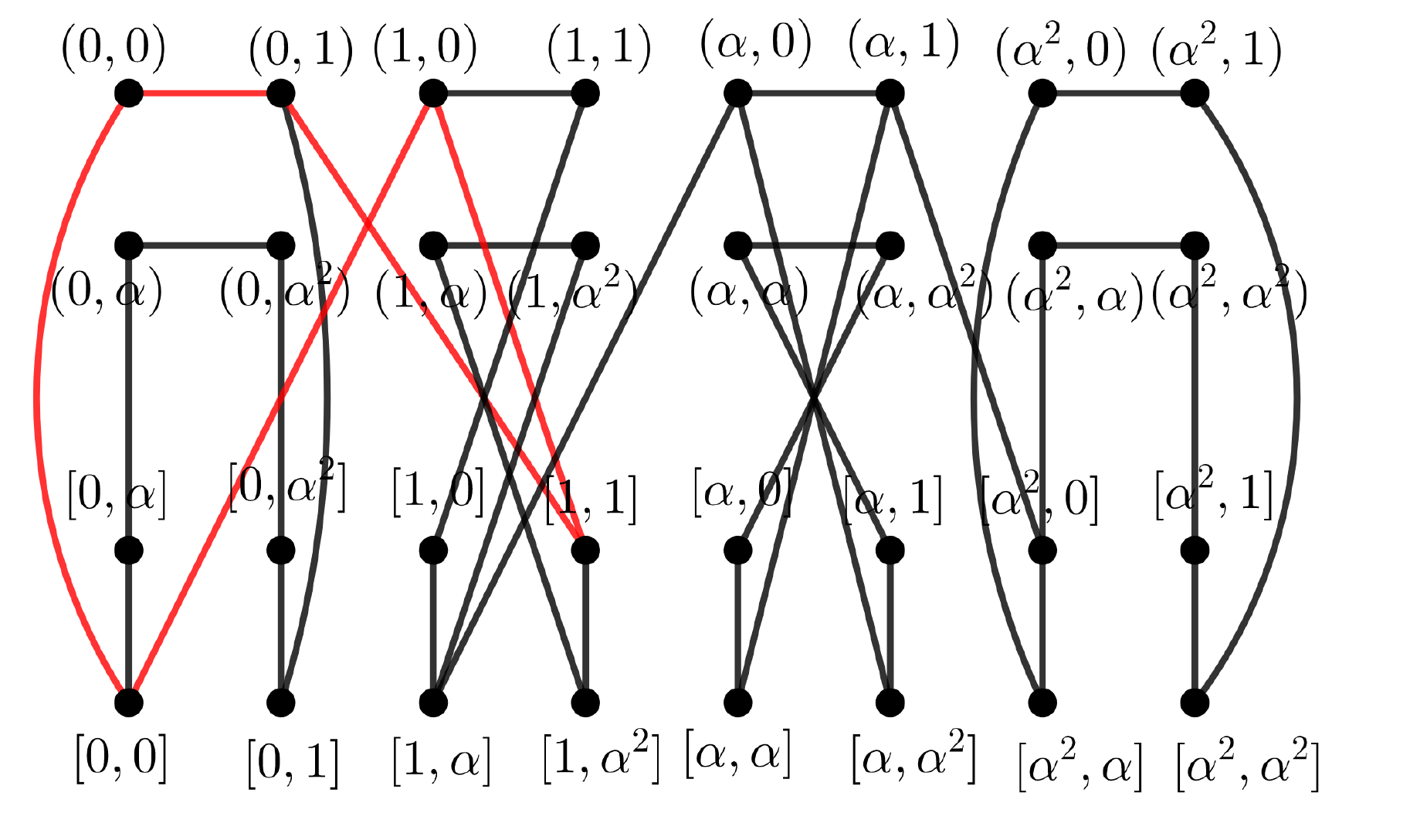}
\caption{A part of the $egr(32,5,5,12)$ graph $\mathcal H$, with some of the edges}
\label{special}
\end{figure}

To prove that $\lambda=12$ we have to analyze three types of edges, namely edges joining vertices both in $\mathcal P$ or both in $\mathcal L$, and edges joining a vertex in $\mathcal P$ to a vertex in $\mathcal L$.
    Note that, by construction, the $5$-cycles do not have more than one edge in $\mathcal E$. Moreover, as $\mathcal B_4$ is bipartite they have exactly one edge on $\mathcal E$.
    
    Let us first consider the case where $a,b\in \mathcal P$ and $ab\in E(\mathcal H)$. Then $a=(x,y)$ and $b=(x,y+1)$ for some $x,y\in \mathbb{F}_4$. We have four options to choose one neighbour $c$ of $b$ in the $\mathcal L$-vertices (we can pick one in each $L_j$ for $j\in \mathbb{F}_4$). Suppose that $bc\in E(\mathcal H)$ and $c$ is a $\mathcal L$-vertex. To take a neighbour $d$ of $c$ in $\mathcal P$ we have three possibilities (one in each $P_x'$ for $x'\in \mathbb{F}_q$ and $x'\not= x$), that is we have the path $(a,b,c,d)$ of length $3$ and now the rest of the cycle is determined by the common neighbour of $a$ and $d$ in $\mathcal L$. Then, we have $12$ options to complete a $5$-cycle that contains $ab$. 

        A similar proof to the one above can be written if $a,b \in \mathcal L$.
        
    Finally, let $u\in \mathcal P$, $v\in \mathcal L$ and $uv\in E(\mathcal H)$. Then $u=(x,y)$ and $v=[m,b]$ with $\{x,y,m,b\}\in \mathbb{F}_4$, and let ${\mathcal C} =(u',u,v,v',w)$ be a $5$-cycle that contains $uv$. As we said before, $\mathcal C$ has exactly one edge in $\mathcal E$. Obviously, $u,v\not\in \mathcal E$.
        Suppose that either $u'u$ or $vv'$ is in $\mathcal E$. If $u'u\in \mathcal E$, we have only one option to take a neighbour $u'$ of $u$ in $P_x$, and we have tree options to take a neighbour $w$ of $u'$ in $P_x'$ for $x'\in \mathbb{F}_4$ and $x\not= x'$. The same happens if $vv'\in \mathcal E$. This gives six circuits of length 5 containing $uv$.
        Now, suppose that either $wu'$ or $v'w$ be in $\mathcal E$. If $wu'\in \mathcal E$, we have three options to select $u'$ (the neighbour of $u$) in $L_m'$ for $m'\in \mathbb{F}_4$ and $m\not= m'$, and we have only one option to select the neighbour of $u'$ in $L_m'$ (that is the edge $wu'\in \mathcal E$) and the rest of the cycle is determined by $v'$ (the neighbour of $v$ and $w$ in $P$). We can make a similar proof if $v'w$ be in $\mathcal E$. This gives another six circuits of length 5 containing $uv$.
        Hence in total we have 12 circuits of length 5 containing the edge $uv$.
\end{proof}


\subsection{Analysing the graph with {\sc Magma}}
Let $\mathcal G$ be the $egr(32,5,5,12)$ that we constructed in the previous section. It can be constructed in {\sc Magma}~\cite{magma} using the following code.

\begin{verbatim}
gr:=Graph<32|
[{ 2, 24, 17, 20, 21 },{ 22, 1, 25, 18, 30 },{ 23, 26, 6, 17, 30 },
{ 14, 26, 27, 19, 31 },{ 25, 27, 17, 28, 9 },{ 3, 28, 18, 31, 21 },
{ 22, 27, 29, 20, 10 },{ 22, 16, 17, 31, 32 },{ 26, 5, 18, 20, 32 },
{ 24, 25, 7, 19, 32 },{ 23, 13, 24, 27, 18 },{ 22, 23, 15, 19, 21 },
{ 11, 28, 19, 30, 20 },{ 23, 4, 28, 29, 32 },{ 12, 24, 29, 30, 31 },
{ 25, 26, 29, 8, 21 },{ 1, 3, 5, 8, 19 },{ 11, 2, 6, 29, 9 },
{ 12, 13, 4, 17, 10 },
{ 1, 13, 7, 9, 31 },{ 1, 12, 16, 27, 6 },{ 12, 2, 28, 7, 8 },
{ 11, 12, 3, 14, 25 },
{ 11, 1, 15, 26, 10 },{ 23, 2, 5, 16, 10 },{ 24, 3, 4, 16, 9 },
{ 11, 4, 5, 7, 21 },{ 22, 13, 14, 5, 6 },{ 14, 15, 16, 7, 18 },
{ 2, 13, 3, 15, 32 },
{ 4, 15, 6, 8, 20 },{ 14, 8, 30, 9, 10 }]>;
\end{verbatim}
It is then easy to check that indeed this graph is 5-regular (that can be seen from the construction) and has girth five using {\sc Magma} built-in functions.
To check that every edge is on 12 pentagons, one can write a program that will construct all pentagons of that graph and count on how many pentagon each edge is.
Such a brute-force program takes less than 2/10th of a second to run and get that $\lambda = 12$. We provide in an appendix at the end of the paper a {\sc Magma} function to check this.

The automorphism group of $\mathcal G$ is $E_{16}:C_3$. It is the semi-direct product of an elementary abelian group of order 16 with a cyclic group of order 3.

{\section{A family of  $egr(q(q^2+1),q,5,\lambda)$-graphs}\label{egrgsg}}

In~\cite{Lee98}, Leemans classified rank two incidence geometries that satisfy some properties for the infinite family of finite simple groups $Sz(q)$.
Among the geometries found, one is a family of graphs that are of interest for this paper. It is the first geometry of~\cite[Table 2]{Lee98}. We recall here how to construct it. But first, we need to give a small introduction to incidence geometry and coset geometries.

\subsection{Incidence geometries and coset geometries}
An {\em incidence system} \cite{BuekCohen}, $\Gamma := (X, *, t, I)$ is a 4-tuple such that
\begin{itemize}
\item $X$ is a set whose elements are called the {\em elements} of $\Gamma$;
\item $I$ is a set whose elements are called the {\em types} of $\Gamma$;
\item $t:X\rightarrow I$ is a {\em type function}, associating to each element $x\in X$ of $\Gamma$ a type $t(x)\in I$;
\item $*$ is a binary relation on $X$ called {\em incidence}, that is reflexive, symmetric and such that for all $x,y\in X$, if $x*y$ and $t(x) = t(y)$ then $x=y$.
\end{itemize}
The {\em incidence graph} of $\Gamma$ is the graph whose vertex set is $X$ and where two vertices are joined provided the corresponding elements of $\Gamma$ are incident. 

A {\em flag} is a set of pairwise incident elements of $\Gamma$, i.e. a clique of its incidence graph.
The {\em type} of a flag $F$ is $\{t(x) : x \in F\}$.
A {\em chamber} is a flag of type $I$.
An element $x$ is {\em incident} to a flag $F$ and we write $x*F$ for that, when $x$ is incident to all elements of $F$.
An incidence system $\Gamma$ is a {\em geometry} or {\em incidence geometry} if every flag of $\Gamma$ is contained in a chamber (or in other words, every maximal clique of the incidence graph is a chamber).
The {\em rank} of $\Gamma$ is the number of types of $\Gamma$, namely the cardinality of $I$.

Let $\Gamma:= (X, *, t, I)$ be an incidence system.
Given $J\subseteq I$, the {\em $J$--truncation} of $\Gamma$ is the incidence system $\Gamma^J := (t^{-1}(J), *_{\mid t^{-1}(J)\times t^{-1}(J)}, t_{\mid J}, J)$. In other words, it is the subgeometry constructed from $\Gamma$ by taking only elements of type $J$ and restricting the type function and incidence relation to these elements.

Let $\Gamma:= (X, *, t, I)$ be an incidence system.
Given a flag $F$ of $\Gamma$, the {\em residue} of $F$ in $\Gamma$ is the incidence system $\Gamma_F := (X_F, *_F, t_F, I_F)$ where
\begin{itemize}
\item $X_F := \{ x \in X : x * F, x \not\in F\}$;
\item $I_F := I \setminus t(F)$;
\item $t_F$ and $*_F$ are the restrictions of $t$ and $*$ to $X_F$ and $I_F$.
\end{itemize}

An incidence system $\Gamma$ is {\em residually connected} when each residue of rank at least two of $\Gamma$ has a connected incidence graph. It is called {\em firm} (resp. {\em thick}) when every residue of rank one of $\Gamma$ contains at least two (resp. three) elements.  

Let $\Gamma:=(X,*, t,I)$ be an incidence system.
An {\em automorphism} of $\Gamma$ is a mapping
$\alpha:(X,I)\rightarrow (X,I):(x,t(x)) \mapsto (\alpha(x),t(\alpha(x))$
where
\begin{itemize}
\item $\alpha$ is a bijection on $X$ inducing a bijection on $I$;
\item for each $x$, $y\in X$, $x*y$ if and only if $\alpha(x)*\alpha(y)$;
\item for each $x$, $y\in X$, $t(x)=t(y)$ if and only if $t(\alpha(x))=t(\alpha(y))$.
\end{itemize}
An automorphism $\alpha$ of $\Gamma$ is called {\em type preserving} when for each $x\in X$, $t(\alpha(x))=t(x)$.
The set of all automorphisms of $\Gamma$ together with the composition forms a group that is called the {\em automorphism group} of $\Gamma$ and denoted by $Aut(\Gamma)$.
The set of all type-preserving automorphisms of $\Gamma$ is a subgroup of $Aut(\Gamma)$ that we denote by $Aut_I(\Gamma)$.
An incidence system $\Gamma$ is {\em flag-transitive} if $Aut_I(\Gamma)$ is transitive on all flags of a given type $J$ for each type $J \subseteq I$.

A rank two geometry with points and lines is called a {\em generalized digon} if every point is incident to every line.

Let $\Gamma$ be a firm, residually connected and flag-transitive geometry.  As defined in~\cite{Bue79}, the \emph{Buekenhout diagram} of $\Gamma$ is a graph whose vertices are the elements of $I$ and with 
an edge $\{i,j\}$ with label $d_{ij} - g_{ij} - d_{ji}$ whenever every residue of type $\{i,j\}$ is not a generalized digon. The number $g_{ij}$ is called the {\em gonality} and is equal to half the girth of the incidence graph of a residue of type $\{i,j\}$. The number $d_{ij}$ is called the $i$-diameter of a residue of type $\{i,j\}$ and is the longest distance from an element of type $i$ to any element in the incidence graph of the residue.
Moreover, to every vertex $i$ is associated a number $s_i$, called the $i$-order, which is equal to the size of a residue of type $i$ minus one, and a number $n_i$ which is the number of elements of type $i$ of the geometry.
If $g_{ij} = d_{ij} = d_{ji} = n$, then every residue of type $\{i,j\}$ of $\Gamma$ is called a {\it generalized
n-gon} and we do not write $d_{ij}$ and $d_{ji}$ on the picture.
If $g_{ij} = d_{ij} = d_{ji} - 1 = 3$, and $s_i = 1$, then every residue of type $\{i,j\}$ of $\Gamma$ is the incidence geometry corresponding to the complete graph of $s_j+1$ vertices and we write $\sf c$ instead of $d_{ij} - g_{ij} - d_{ji}$ on the picture.

The basic concepts about geometries constructed from a group and some of its subgroups are due to Jacques Tits \cite{Tit57}
(see also \cite{Bue95}, chapter 3).
Let $I$ be a finite set and let $G$ be a group together with a family of subgroups ($G_i$)$_{i \in I}$.
We define the incidence system $\Gamma = \Gamma(G,(G_i)_{i \in I})$ as follows.
The set $X$ of {\it elements} of $\Gamma$ consists of all cosets $G_ig$, $g \in G$, $i \in I$.
We define an {\it incidence relation} * on $X$ by :
$G_ig_1$ * $G_jg_2$ iff $G_ig_1 \cap G_jg_2 \neq \emptyset$.
The {\it type function} $t$ on $\Gamma$ is defined by $t(G_ig) = i$.
The {\it type} of a subset $Y$ of $X$ is the set $t(Y)$;
its {\it rank} is the cardinality of $t(Y)$ and we call $\mid t(X) \mid$ the {\it rank} of $\Gamma$.
The incidence system $\Gamma$ is also called a {\it coset geometry} as it is build from cosets of subgroups of a group.
The group $G$ acts on $\Gamma$ as an automorphism group, by right translation, preserving the type of each element.\\
Every coset geometry of rank at most three is a geometry.

Let $\Gamma(G;G_0, \ldots ,G_{n-1})$ be a rank $n$ pre-geometry.
We call $C = \{G_0, \ldots ,G_{n-1}\}$ the {\it maximal parabolic chamber associated to} $\Gamma$.
Assuming that $F$ is a subset of $C$,
the {\it residue} of $F$ is the pre-geometry
\begin{center}
$\Gamma_F = \Gamma(\cap_{j \in t(F)} G_j,(G_i \cap(\cap_{j \in t(F)}G_j))_{i \in I \backslash t(F)})$
\end{center}
If $\Gamma$ is flag-transitive and $F$ is any flag of $\Gamma$, of type $t(F)$, then the residue
$\Gamma_F$ of $\Gamma$ is isomorphic to the residue of the flag $\{G_i, i \in t(F)\} \subseteq C$.

Given a rank two coset geometry $\Gamma(G, \{G_0,G_1\}$
, the {\em 0-incidence graph} or {\em point-incidence graph} of $\Gamma$ is the graph whose vertices are the orbits of the action of $G_0$ on the cosets of $G_0$ and $G_1$. An orbit $A$ is joined to an orbit $B$ provided there are cosets in $A$ that are incident to cosets in $B$. Moreover, on the edge, we add two numbers, one of each closer to one of the vertices. The number $X$ near the orbit $A$ tells us how many cosets in the orbit $B$ are incident to a given coset of the orbit $A$, and the number $Y$ near the vertex $B$ tells us how many cosets in the orbit $A$ are incident to a given coset of the orbit $B$. Finally, inside a vertex, we put the number of elements of the orbit the vertex represents. Usually, we use two types of vertices to distinguish easily from which types of cosets the orbits are.
We can define the {\em 1-incidence graph} or {\em line-incidence graph} by swapping the roles of 0 and 1 in the above.


\subsection{A family of $egr(q(q^2+1),q,5,q-1)$-graphs}
Let $G\cong Sz(q)$ (with $q = 2^{2e+1}$ and $e$ a strictly positive integer) be a Suzuki group acting on an ovoid $\mathcal D$ of $PG(3,q)$.
Let $p$ be a point of $\mathcal D$ and $G_p$ be the stabilizer of $p$ in $G$.
The subgroup $G_p\cong E_q\cdot E_q:C_{q-1}$ has a class of $q$ maximal subgroups isomorphic to $E_q : C_{q-1}$.
Let $G_1 \cong E_q : C_{q-1}$ be the stabilizer of a circle of $\mathcal D$. It has a special point $n$ called the nucleus. Its orbits on $\mathcal D$ are of respective sizes $[1,q,q^2-q]$. Pick a point $p$ in the orbit of size $q$. Let $G_2$ be the stabilizer of $\{n,p\}$ in $G$. Then $G_2\cong D_{2(q-1)}$ is such that $G_1\cap G_2 \cong C_{q-1}$.
The coset geometry $\Gamma(G;\{G_1,G_2\})$ is such that the index of $G_1\cap G_2$ in $G_2$ is two, meaning we can construct a graph from this geometry. Let $\mathcal G_q(V,E)$ be the graph whose vertex-set $V$ is the set of cosets of $G_1$ in $G\cong Sz(q)$ and whose edge-set is the set of cosets of $G_2$ in $G$. A vertex $G_1g$ is on an edge $G_2h$ if and only if $G_1g\cap G_2h \neq \emptyset$.

\begin{lemma}\label{girthatmostfive}
The graph $\mathcal G_q(V,E)$ has girth at most five.
\end{lemma}
\begin{proof}
This is a direct consequence of the fact that $\Gamma(G;\{G_1,G_2\})$ is a rank two truncation of a geometry of rank three of the second family described in~\cite[Theorem 5.3]{Lee99}.
One can check that the geometry $\Gamma := \Gamma(G;\{G_1,G_2, G_3\})$ with $G_3 \cong D_{10}$ such that $G_1\cap G_3\cong C_2 \cong G_2\cap G_3$ and $G_1\cap G_2\cap G_3\cong \{1_G\}$ exists even when $q-1$ is not a prime. In that case, $\Gamma$ is not residually weakly primitive but this is not important for our purpose. The geometry $\Gamma$ has the following Buekenhout diagram (see~\cite{Lee99} for the definitions).

\begin{center}
\begin{picture}(10,2)
\put(0,1.8){\circle*{0.2}}
\put(5,1.8){\circle*{0.2}}
\put(10,1.8){\circle*{0.2}}
\put(0,1.8){\line(1,0){10}}
\put(2.5,2){5}
\put(7.5,2){$\mathsf c$}
\put(0,1.2){1}
\put(5,1.2){1}
\put(10,1.2){$q-2$}
\put(0,0.6){$(q^2+1)q^2$}
\put(5,0.6){$(q^2+1)q$}
\put(10,0.6){$(q^2+1)q^2(q-1)/10$}
\put(0,0){$D_{2(q-1)}$}
\put(5,0){$E_q:C_{q-1}$}
\put(10,0){$D_{10}$}
\end{picture}
\end{center}
The geometry $\Gamma(G;\{G_1,G_2\})$ being a truncation of $\Gamma$ and the rank two residues of type $\{1,2\}$ being pentagons, we can conclude that the girth of $\mathcal G_q(V,E)$ is at most five.
\end{proof}

\begin{lemma}\label{girthfive}
The graph $\mathcal G_q(V,E)$ has girth five.
\end{lemma}
\begin{proof}
A careful analysis of the point and line incidence graphs of $\Gamma : = \Gamma(G;\{G_1,G_2\})$ shows that the gonality of $\Gamma(G;\{G_1,G_2\}$ has to be at least five. Let us call the cosets of $G_1$ the points and the cosets of $G_2$ the lines.
There are $[G:G_1] = (q^2+1)q$ points and $[G:G_2] = (q^2+1)q^2/2$ lines.
The point-incidence graph of $\Gamma$ starts as follows thanks to the two-transitive action of $G$ on the ovoid.

\begin{center}
\begin{picture}(10,1)
\put(-0.44,0.5){\framebox{1}}
\put(0.1,0.8){$q$}
\put(0,0.6){\line(1,0){2}}
\put(2.3,0.6){\circle{0.6}}
\put(2.2,0.6){$q$}
\put(1.8,0.8){$1$}
\put(2.6,0.6){\line(1,0){2}}
\put(2.6,0.8){$1$}
\put(4.6,0.5){\framebox{$q$}}
\put(4.3,0.8){$x$}
\multiput(5,0.6)(0.2,0){5}{\line(1,0){0.1}}
\end{picture}
\end{center}
We know that $x$ is equal to 1 for, if it were greater than 1, it would contradict the fact that through two distinct points there is exactly one line in a projective space.
Now, because $C_{q-1}$ is the stabilizer of two points, there is a unique orbit of lines at distance 3 of the starting point.
The point-incidence graph can therefore be extended as follows.
\begin{center}
\begin{picture}(10,1)
\put(-0.44,0.5){\framebox{1}}
\put(0.1,0.8){$q$}
\put(0,0.6){\line(1,0){2}}
\put(2.3,0.6){\circle{0.6}}
\put(2.2,0.6){$q$}
\put(1.8,0.8){$1$}
\put(2.6,0.6){\line(1,0){2}}
\put(2.6,0.8){$1$}
\put(4.6,0.5){\framebox{$q$}}
\put(4.3,0.8){$1$}
\put(5,0.6){\line(1,0){2}}
\put(7.8,0.6){\oval(1.6,1)}
\put(7.2,0.6){$q(q-1)$}
\put(5.1,0.8){$q-1$}
\multiput(8.6,0.6)(0.2,0){5}{\line(1,0){0.1}}
\put(6.8,0.8){$y$}
\end{picture}
\end{center}
Here, $y$ cannot be equal to 2 for otherwise, we would have $q+q(q-1)$ lines, which is not enough. Hence $y = 1$ and this already shows that the girth of the graph is at least 4.
We can also draw the line-incidence graph of the geometry. It starts as follows.
\begin{center}
\begin{picture}(10,1)
\put(0,0.6){\circle{0.6}}
\put(-0.1,0.6){$1$}
\put(0.3,0.6){\line(1,0){2}}
\put(0.3,0.8){$2$}
\put(2.3,0.5){\framebox{$2$}}
\put(2.8,0.8){$q-1$}
\put(2,0.8){$1$}
\put(2.7,0.6){\line(1,0){2}}
\put(5.5,0.6){\oval(1.6,1)}
\put(4.9,0.6){$2(q-1)$}
\put(4.5,0.8){$1$}
\put(6.3,0.6){\line(1,0){2}}
\put(0.3,0.8){$2$}
\put(8.3,0.5){\framebox{$2(q-1)$}}
\put(6.4,0.8){$1$}
\put(7.8,0.8){$1$}
\multiput(10,0.6)(0.2,0){5}{\line(1,0){0.1}}
\end{picture}
\end{center}
This graph continues on the right with $q-1$ branches having a weight of 1 each. Each of these branches leads to an orbit of $2(q-1)$ lines, each of them being adjacent to exactly one of the $2(q-1)$ points of the box at the right of the picture. Hence the girth is at least five.
By Lemma~\ref{girthatmostfive}, we can conclude that the girth is indeed five.
\end{proof}
\begin{theorem}
The family of graphs $\mathcal G_q(V,E)$ is a family of
egr-graphs of girth five.
\end{theorem}
\begin{proof}
For a give $q=2^{2e+1}$, the graph $\mathcal G_q(V,E)$ is a regular graph with $q(q^2+1)$ vertices since $[G:G_1] = q(q^2+1)$. Its degree is $[G_1:G_1\cap G_2] = q(q-1)/(q-1) = q$.
By Lemma~\ref{girthfive}, we already know that the girth is five.
As the graph is arc-transitive, it has to be an edge-girth-regular graph.
\end{proof}

The only unknown parameter for these graphs is $\lambda$.
We know that $\lambda$ is at least $q-1$. This is an immediate consequence of the geometry $\Gamma$ of rank three described in the proof of Lemma~\ref{girthatmostfive}. Indeed, the circuits of length five are the elements of type 3 of that geometry. Each edge of the graph is incident to $[G_2:G_2\cap G_3]$ pentagons of that geometry, that is $2(q-1)/2 = q-1$ pentagons. Hence $\lambda \geq q-1$. We conjecture that $\lambda = q-1$.
\begin{conjecture}\label{lambda=q-1}
The family of graphs $\mathcal G_q(V,E)$ is a family of
$egr(q(q^2+1),q,5,q-1)$-graphs.
\end{conjecture}
We managed to check that conjecture with {\sc Magma} for $q=8$.

\section{Extremal Edge-girth-regular graphs}\label{extremalegrg}

As we said in the introduction, recently Drglin, Filipovski, Jajcay and Raiman~\cite{DFJR21} introduced a special class of edge-girth-regular graphs called {\it{extremal edge-girth-regular graphs}}. These are $egr(v,k,g,\lambda)$ with $v=n(k,g,\lambda)$ where $n(k,g,\lambda)$ is the smallest order of a $(v,k,g,\lambda)$-graph fixing the triplet $(k,g,\lambda)$. 
Recall that, for an $egr(v,k,g,\lambda)$ graph $\mathcal G$, the {\emph{excess}} of $\mathcal G$ is the difference  $|V(\mathcal G)|$-$n(k,g,\lambda)$.

In this section we state the lower bounds for extremal edge-girth-regular-graphs, we recall some edge-girth-regular graphs that are also extremal (the Pappus graph and others given also by Jajcay in \cite{DFJR21}) and prove that the graphs given in Section \ref{sub1} are extremal. In particular, for $q=3$, we obtain the Pappus graph. Also we give the excess for the other graphs constructed in this paper. 
In fact, we conjecture that the construction given in Section \ref{sub1} produces "small" girth-regular graphs that is graphs with order "close" to be extremal. In particular, we analyze the $(32,5,5,12)$-graph given in Section \ref{sub2} that should be an extremal $(32,5,5,12)$-graph. 

The lower bound for extremal edge-girth-regular graphs given in \cite{DFJR21} is the following.
\begin{theorem}\cite{DFJR21}\label{extremalbound}
Let $k$ and $g$ be a fixed pair of integers greater than or equal to $3$, and let $\lambda\leq (k-1)^{\frac{g-1}{2}}$, when $g$ is odd and $\lambda \leq (k-1)^{\frac{g}{2}}$, when $g$ is even. Then: 

\begin{equation}\label{lowercages} n(k,g,\lambda)  \geq n_0(k,g) + \left\{ \begin{array}{ll} (k-1)^{\frac{g-1}{2}}- \lambda , &\mbox{ if $g$ is odd};\\
 \lceil 2 \frac{(k-1)^{\frac{g-1}{2}}- \lambda}{k}\rceil, &\mbox{ if $g$ is
even}.\end{array}\right.\end{equation}

\end{theorem}

 Moreover, in the same paper, the authors improve the lower bound for bipartite graphs and give the following Theorem: 
\begin{theorem}\cite{DFJR21}\label{extremalbipartitebound}
Let $k\geq 3$ and $g\geq 4$ be even, and let $\lambda\leq (k-1)^{\frac{g-1}{2}}$. If $\mathcal G$ is a bipartite $egr(v,k,g,\lambda)$-graph, then: 

$$v  \geq n_0(k,g) + 2 \lceil  \frac{(k-1)^{\frac{g-1}{2}}- \lambda}{k}\rceil$$
 
 \end{theorem}

We now use Theorem \ref{extremalbipartitebound} and the fact that $\mathcal B_q$ is a bipartite graph to prove the following theorem. 

\begin{theorem}
For $q$ a prime power, 
the graph $\mathcal B_q$ defined in Remark~\ref{remark1} is 
extremal.
\end{theorem}
\begin{proof}
Let $\mathcal B_q$ be the graph constructed in Remark~\ref{remark1}.
As pointed out in Theorem~\ref{B_q}, this graph is an $egr(2q^2,q,6,(q-1)^2(q-2))$.
Theorem~\ref{extremalbipartitebound} gives that
$$n(q,6,(q-1)^{2}(q-2))\geq 2(q^2-q+1)+2\lceil\frac{(q-1)^2}{q}\rceil=2q^2.$$
The order of $\mathcal B_q$ is exactly $2q^2$. Hence it is an extremal edge-girth-regular graph. 
\end{proof}

Observe that the graph $\mathcal B_3$, that is an $egr(18,3,6,4)$, is the Pappus Graph.

Let us come back to the special $egr(32,5,5,12)$-graph, constructed in Section~\ref{sub2}.
The bound given in Theorem \ref{extremalbound} implies that  $n(5,5,12)\geq 30$ and therefore the $egr(32,5,5,12)$-graph has excess at most $2$, but as we said in the introduction, the order of a $(5,5)$-cage is exactly $30$, as $n_0(5,5)=26$, if $\mathcal H$ is a $(5,5)$-cage, $\mathcal H$ has excess $4$, and there exist four different cages with excess $4$ (see \cite{EJ13}). Then, to make sure that our construction gives an extremal graph we have to prove that none of the four $(5,5)$-cages are edge-girth-regular-graphs and that there is no $(5,5)$-graph or order $31$.


\begin{theorem}
The $egr(32,5,5,12)$-graph defined in Section~\ref{sub2} is extremal.
\end{theorem}
\begin{proof}
There are four Moore cages of type (5,5) and order 30. These are the only possible graphs that would have a smaller order than the $egr(32,5,5,12)$-graph and therefore make it not extremal. These graphs are available for instance from Andries Brouwer's website (see \url{https://www.win.tue.nl/~aeb/graphs/cages/cages.html}).
An easy and quick computation with {\sc Magma}~\cite{magma} then permits to check that none of these four graphs are egr-graphs. One way to do this is to use the {\sc Magma} function we provide in the appendix at the end of the paper. Following Brouwer's website ordering, the first and third graphs have edges that belong to 12 pentagons and edges that belong to 14 pentagons, the second has edges that belong to 12 pentagons, edges that belong to 13 pentagons and edges that belong to 14 pentagons and, finally, the fourth has edges that belong to 12 pentagons, edges that belong to 13 pentagons and edges that belong to 16 pentagons.

The fact that there is no $(5,5)$-graph of order 31 is due to the fact that the degree of each vertex is 5 and therefore the order has to be even by the handshaking lemma (the number of vertices times 5 is equal to the number of edges times 2).
\end{proof}

Finally we note that, as $n(q,5,q-1)\geq 2q^2-3q+3$ the graphs constructed in Section \ref{egrgsg} have excess at most $q^3-2q^2+4q-3$ if the Conjecture \ref{lambda=q-1} is true and $\lambda=q-1$. In particular, it is true for $q=8$.

We conclude this paper with two conjectures.

Taking into account the fact that edge-regular-graphs should have a lot of symmetries, that the family of graphs given on Remark~\ref{remark1} induces extremal edge-regular-graphs and that there exists a generalization of the construction of $\mathcal B_q$ for girth $g=\{8,12\}$, given in \cite{AMGS07}, using the incidence graphs of generalized quadrangles and hexagons (see \cite{BBR18}), or the $(k,g)$-Moore cages for $g=\{8,12\}$, give us $(k,g)$-graphs for $k$ a prime power and order $2k^{\frac{g-2}{2}}$, we conjecture the following. 

\begin{conjecture}\label{8y12}
For $q\geq 3$ a prime power and $g=\{8,12\}$ there exist a family of $egr(2q^{\frac{g-2}{2}},q,g,(q-1)^{\frac{g-2}{2}}(q-2))$-graphs.
\end{conjecture} 
\begin{conjecture}
The family of graphs conjectured in \ref{8y12} are extremal-edge-girth-regular graphs. \end{conjecture} 

\section{Appendix}
We give here the little piece of {\sc Magma} code we wrote to check whether the graphs we mention in this paper are egr-graphs. This function receives as input a graph and returns either false if the graph is not an egr-graph, or true and the parameters $v$, $k$, $g$, $\lambda$ if the graph is an egr-graph.

\begin{verbatim}
function IsEdgeGirthRegular(gr);
  n:=Order(gr);
  g:=Girth(gr);
  V:=VertexSet(gr);
  E:=EdgeSet(gr);
  c2:= [[i,j] : j in {i+1..n},i in {1..n}| {V!i,V!j} in E];
  if IsRegular(gr) then
    k := #Neighbours(V!1);
    while #c2 ne 0 and #Rep(c2) le g-1 do
      cNext:=[];
      for x in c2 do
        for j := 1 to n do
          if x[#x] ne j and {V!x[#x],V!j} in E and not(j in Set(x)) 
            and (#x lt g-1 or {V!x[1],V!j} in E) then
            Append(~cNext,x cat [j]);  
          end if;
        end for;
      end for;
      c2 := cNext;
    end while;
    c2 := [x cat [x[1]] : x in c2];
    m := {* {y[i],y[i+1]} : i in [1..g], y in c2 *};
    m2 := {Multiplicity(m,x)/g : x in m};
    if #m2 eq 1 then 
      return true, #V, k, g, Rep(m2);
    end if;
  end if;
  return false;
end function; 
\end{verbatim}


\begin{thebibliography}{99}

\bibitem{AABB17} E. Abajo, G. Araujo-Pardo, C. Balbuena and M. Bendala, New small regular graphs of girth five, {\em Discrete Math.} \textbf{340}, no. 8 (2017), 1878--1888.

\bibitem{AABL12} M. Abreu, G. Araujo-Pardo, C. Balbuena and D. Labbate, Families of small regular graphs of girth 5, {\em Discrete Math.} {\bf 312} (2012), 2832--2842.

\bibitem{ANS06} G. Araujo, M. Noy and O. Serra, A geometric construction of large vertex transitive graphs of diameter 2, {\em Journal of Combinatorial Mathematics and Combinatorial Computing} {\bf 57} (2006), 97--102.

\bibitem{AMGS07} G. Araujo, D. Gonz\'alez-Moreno, J.J. Montellano-Ballesteros and O. Serra, On upper bounds and connectivity of cages, \emph{Australas. J. Combin.}, {\bf 38} (2007), 221--228.

\bibitem{ABMM16} G. Araujo-Pardo, C. Balbuena, M. Miller and M. \v{Z}d{\' i}malov{\' a}, A family of mixed graphs with large order and diameter 2, {\em Discrete Applied Math.} \textbf{218} (2017), 57--63.

\bibitem{ADG20} G. Araujo-Pardo, C. De la Cruz and D. González-Moreno, Mixed Cages: monotony, connectivity and upper bounds, Preprint. arXiv:2009.13709.

\bibitem{BBR18}
J. Bamberg, A. Bishnoi and G.F. Royle,
On regular induced subgraphs of generalized polygons,
{\em J. Comb. Theory,} Ser. A {\bf 158} (2018), 254--275.




\bibitem {BI80} N. L. Biggs and T. Ito, Graphs with even girth and small excess, \emph {Math. Proc. Camb. Philos. Soc.} \textbf{88} (1980), 1--10.

\bibitem{magma}
W.~Bosma, J.~Cannon, and C.~Playoust.
\newblock The {M}agma {A}lgebra {S}ystem {I}: the user language.
\newblock {\em J. Symbolic Comput.} {\bf (3/4)} (1997), 235--265.

\bibitem{B67} W.G. Brown, On Hamiltonian regular graphs of girth six, {\em J. London Math. Soc.} {\bf 42 } (1967), 514--520.

\bibitem{Bue79}
F. Buekenhout,
\newblock Diagrams for geometries and groups,
\newblock {\em J. Combin. Theory Ser. A} {\bf 27} (1979), 121--151.

\bibitem{Bue95}
F.~Buekenhout, editor,
\newblock {\em Handbook of Incidence Geometry. {B}uildings and Foundations},
\newblock Elsevier, Amsterdam, 1995.

\bibitem{BuekCohen}
F. Buekenhout and A.~M. Cohen,
\newblock {\em Diagram Geometry. Related to classical groups and buildings},
\newblock Ergebnisse der Mathematik und ihrer Grenzgebiete. 3. Folge. A Series of Modern Surveys in Mathematics [Results in Mathematics and Related Areas. 3rd Series. A Series of Modern Surveys in Mathematics], 57. Springer, Heidelberg, 2013. xiv+592 pp.

\bibitem{DFJR21} A. Z. Drglin, S. Filipovski, R. Jajcay and T. Raiman, Extremal Edge-Girth-Regular Graphs, \emph{Graphs and Combin.} https://doi.org/10.1007/s00373-021-02368-9

\bibitem{ES63} P. Erd\"os and H. Sachs, Regul\"are Graphen gegebener Taillenweite mit minimaler Knotenzahl, {\em Wiss. Z. Uni. Halle (Math. Nat.)} \textbf{12} (1963), 251--257 .

\bibitem{EJ13} G. Exoo and R. Jajcay, Dynamic Cage Survey. \emph{Electron. J.
  Combin}, Dynamic Survey \#DS16: Jul 26, 2013.
  
\bibitem{ES99} L. Eroh and A. Schwenk, Cages of girth 5 and 7, \emph{Congr. Numer.} \textbf{138} (1999), 157--173.


\bibitem{FJ18} S. Filipovski and R. Jajcay, On the excess of vertex-transitive graphs of given degree and girth, \emph{Discrete Math.} \textbf{341} (2018), 772--780.




\bibitem{JKM18} R. Jajcay, G. Kiss and S. Miklavi\v{c}, Edge-girth-regular graphs, {\em{European Journal of Combinatorics}} {\bf 72} (2018), 70--82.


\bibitem{Lee98}
D.~Leemans,
\newblock The rank 2 geometries of the simple {S}uzuki groups {S}z(q),
\newblock {\em Beitr{\"a}ge Algebra Geom.} {\bf 39}, no. 1 (1998), 97--120.

\bibitem{Lee99}
D.~Leemans,
\newblock The rank 3 geometries of the simple {S}uzuki groups {S}z(q),
\newblock {\em Note Mat.}, {\bf 19}, no. 1 (1999), 43--63.

\bibitem{M98} H. Van Maldeghem, \emph{Generalized Polygons}, Birkhauser, Basel, 1998.
  
  
\bibitem{PW07} P. Poto\v{c}nik and S. Wilson. Tetravalent edge-transitive graphs of girth at most 4. \emph{J. Combin. Theory Ser. B} \text{92 (2)} (2007), 217--236.
  
\bibitem{Tit57}
J.~Tits,
\newblock Sur les analogues alg\'ebriques des groupes semi-simples complexes,
\newblock in {\em Colloque d'alg\`ebre sup\'erieure, tenu \`a Bruxelles du 19
  au 22 d\'ecembre 1956}, Centre Belge de Recherches Math\'ematiques, pages
  261--289. \'Etablissements Ceuterick, Louvain, 1957.

 \bibitem{T47} W. T. Tutte, A family of cubical graphs, {\em Math. Proc. Cambridge Philos. Soc.} \textbf{43}, no. 4 (1947), 459--474.
 
\end{thebibliography}
\end{document}